\documentclass[a4paper,12pt]{article}
\usepackage[utf8x]{inputenc}

\usepackage{amsmath,amsthm,amssymb}
\usepackage{mathrsfs}
\usepackage{enumerate}

\DeclareMathOperator{\Aut}{Aut}

\DeclareMathOperator{\supp}{supp}

\DeclareMathOperator{\cC}{\mathcal C}
\DeclareMathOperator{\cL}{\mathcal L}
\DeclareMathOperator{\Fr}{\rm{Fr}}
\DeclareMathOperator{\cX}{\mathcal X}
\DeclareMathOperator{\Div}{\rm{Div}}
\DeclareMathOperator{\cF}{\mathcal F}

\DeclareMathOperator{\cA}{\mathcal A}
\newtheorem{theorem}{Theorem}[section]
\newtheorem{proposition}[theorem]{Proposition}
\newtheorem{lemma}[theorem]{Lemma}
\newtheorem{corollary}[theorem]{Corollary}

\newtheorem{rem}[theorem]{Remark}

\newcommand{\curve}[1]{\mathbf{v}(#1)}
\newcommand{\HC}{\mathscr{H}}

\title{Hermitian codes from higher degree places}
\author{G.~Korchm\'aros%
\thanks{This research was performed while the first author was a visiting professor at the Bolyai Institute of University of Szeged during the second semester of the academic year 2011-12. The visit was financially supported by the TAMOP-4.2.1/B-09/1/KONV-2010-0005 project.}~~and G.P.~Nagy}
\date{}

\begin{document}

\maketitle

\begin{abstract} Matthews and Michel \cite{Michel} investigated the minimum distances in certain algebraic-geometry codes arising from a higher degree place $P$. In terms of the Weierstrass gap sequence at $P$, they proved a bound that  gives an improvement on the designed minimum distance. In this paper, we consider those of such codes which are constructed from the Hermitian function field $\mathbb F_{q^2}(\HC)$. We determine the Weierstrass gap sequence $G(P)$ where $P$ is a degree $3$ place of $\mathbb F_{q^2}(\HC)$, and compute the Matthews and Michel bound with the corresponding improvement. We show more improvements using a different approach based on geometry. We also compare our results with the true values of the minimum distances of Hermitian $1$-point codes, as well as with estimates due Xing and Chen
\cite{XC}.
\end{abstract}



\section{Introduction}
Algebraic-geometry (AG) codes are linear codes constructed from  algebraic curves defined over a finite field $\mathbb F_{q}$. The best known such general construction was originally introduced by Goppa, see \cite{Go}. It provides linear codes
from certain rational functions whose poles are prescribed by a given $\mathbb F_{q}$-rational divisor $G$, by evaluating them at some set of $\mathbb F_{q}$-rational places disjoint from $\supp(G)$. The dual to such a code can be obtained by computing residues of differential forms. The former are the \emph{functional} codes, and the latter are the \emph{differential} codes. If the $\mathbb F_{q}$-rational places are $Q_1,\ldots,Q_n$ and $D=Q_1+\ldots+Q_n,$ then $C_L(D,G)$ and $C_\Omega(D,G)$ stand for the corresponding functional and differential codes, respectively.
For $n>\deg G> 2g-2$ where $g$ is the genus of the curve, a lower bound on the minimum distance for $C_L(D,G)$ is $n-\deg G$, and for $C_\Omega(D,G)$ is  $\deg G-(2g-2)$. These values are the \emph{designed minimum distance}.

Typically the divisor $G$ is taken to be a multiply $mP$ of a single place $P$ of degree one. Such codes are the  \emph{one-point} codes, and have been extensively investigated; see \cite{bal3,GMRT} and the bibliography therein.
It has been shown however that AG-codes with better parameters than the comparable one-point Hermitian code may be obtained by allowing the divisor $G$ to be more general; see the recent papers \cite{bal1,bal2,cal,di1,di2,GST} and the references therein.

In \cite{Michel} this possibility is discussed for one-point differential codes arising from places of higher degree, that is, for $C_\Omega(D,G)$  with $G=mP$ where $P$ is a place of degree $r>1$. From
\cite[Theorem 3.4]{Michel}, there exist special values of $m$ for which such a code $C_\Omega(D,G)$ has bigger minimum distance than the designed one by at least $r$. The Matthews-Michel bound, see \cite[Theorem 3.5]{Michel}, shows that even better improvements may occur  whenever the gap sequence at $P$ has certain specific properties. This is verified in \cite{Michel} by the examples computed by MAGMA \cite{magma} for $q=7^2,8^2$ and $r=3$ where the curve is, as usual, the Hermitian curve over $\mathbb F_{q^2}$. Nevertheless, the applicability of the above results to any $q$ requires detailed knowledge of the gap sequence at $P$ rising the problem of determining such a sequence, in particular at a degree $3$ point $P$ of the Hermitian curve over $\mathbb F_{q^2}$. Our Theorem \ref{th:HP} solves this problem and together with \cite[Theorem 3.5]{Michel} provides
an improvement on the designed minimum distance for an infinite family of differential codes, see Proposition \ref{pr:MathewsMichel}.
This confirms the importance of knowledge of gap sequences at $r$-tuples of places in the study of functional and differential codes, as clearly emerged from previous and current work by several authors, see \cite{CK,CK2,C1,CT,GKL,Homma,HK,HKK,M1,M2,M3}.


In Section \ref{sectionnoether} we give more improvements using a different approach based on geometry rather than function field theory, the essential ingredient being the Noether ``AF+BG'' theorem. Our main result is stated in
Theorem \ref{th:main}.

In Section \ref{sectionex} examples are given to illustrate and compare the above improvements. For the Hermitian curve over $\mathbb F_{7^2}$ with a point $P$ of degree $r=3$, the Matthews-Michel bound as well as Theorem \ref{th:main}
show that  $C_\Omega(D,18P)$ is a $[343,309,d]$-code with $d\geq 20$. This improves the previous Xing-Chen bound by $2$, see \cite{XC}, and the designed minimum distance by $6$. Indeed, using MAGMA, we were able to prove that such a code has minimal distance $20$.

\section{Background and Preliminary Results} Our notation and terminology are standard. The reader is referred to \cite{HiKoToBook}, \cite{sti} and  the survey paper \cite{hop}.

Let $\cX$ be a (projective, non-singular, geometrically irreducible algebraic curve) of genus $g$, defined over a finite field $\mathbb F_{q}$ of order $q=p^e$ and viewed as curve over the algebraic closure of $\mathbb F_{q}$. Let $\mathbb F_{q}(\cX)$ be the function field of $\cX$  with constant field $\mathbb F_{q}$. For every non-zero function $f\in \mathbb F_{q}(\cX)$, $\Div(f)$ stands for the principal divisor associated with $f$ while $\Div(f)_0$ and $\Div(f)_\infty$ for its zero and pole divisor. Furthermore, for every separable function $f\in\mathbb F_{q}(\cX)$, $df$ is the exact differential arising from $f$, and $\Omega$ denotes the set of all these differentials. Also, ${\rm{res}}_P(df)$ is the residue of $df$ at a place of $P$ of $\mathbb F_{q}(\cX)$. For any divisor $A$ of $\mathbb F_{q}(\cX)$, let
$$\cL(A)=\{f\in \mathbb F_{q}(\cX)\setminus \{0\}|\, \Div(f)\succeq -A \}\cup \{0\}$$
 and $\ell(A)=\dim(\cL(A))$. Furthermore, let
$$\Omega(A)=\{df\in \Omega|\, \Div(df)\succeq A \}\cup \{0\}.$$

Let $D=Q_1+\ldots+Q_n$ be a divisor where $Q_1,\ldots,Q_n$ are $n$ distinct degree one places of $\mathbb F_{q}(\cX)$. Let $G$ be another divisor of $\mathbb F_{q}(\cX)$ whose support $\supp(G)$ contains none of the places $P_i$ with $1\le i \le n$. For any function $f\in \cL(G)$,  the \emph{evaluation} of $f$ at $D$ is given by
${\rm{ev_D}}(f)=(f(Q_1),\ldots f(Q_n))$. This defines the \emph{evaluation map} ${\rm{ev}}_D:\cL(G)\to \mathbb F_{q}^n$ which is $\mathbb F_{q}$-linear and also injective when $n>\deg(G)$. Therefore, its image is a subspace of
the vector space $\mathbb F_{q}^n$, or equivalently, an AG $[n,k,d]$-code where $d\geq n-\deg(G)$ and if $\deg(G)>2g-2$ then $k= \deg(G)+1-g$. Such a code is the \emph{functional} code $C_L(D,G)$ with designed minimum distance $ n-\deg(G)$. The dual code $C_\Omega(D,G)$ of $C_L(D,G)$ is named \emph{differential code}, since
$$C_\Omega(D,G)=\{ ({\rm{res}}(df)_{Q_1},\ldots,{\rm{res}}(df)_{Q_n})|\,df\in\Omega(G-D)\}.$$
The differential code $C_\Omega(D,G)$ is a $[n,\ell(G-D)-\ell(G)+\deg D,d]$-code with $d\geq \deg(G)-(2g-2)$, and its designed minimum distance is  $\deg(G)-(2g-2)$.

In this paper we are interested in differential codes $C_\Omega(D,G)$ with $G=mP$ where $P$ is a degree $r$ place of $\mathbb F_{q}(\cX)$.
Let $P_1,\ldots,P_r$ be the extensions of $P$ in the constant field extension of $\mathbb F_{q}(\cX)$ of degree $r$. Then $P_1,\ldots,P_r$ are degree one places of $\mathbb F_{q^r}(\cX)$ and, up to labeling the indices, $P_{j+1}=\Fr(P_j)$ where $\Fr$ is the $q$-th Frobenius map and the indices are taken modulo $n$. Also, $P$ may be identified with the $\mathbb F_{q}$-divisor $P_1+\ldots+P_r$ of $\mathbb F_{q^r}(\cX)$. The relationship between the Weierstrass semigroups $H(P)$ of $\mathbb F_{q}(\cX)$ and $H(P_1,\ldots,P_r)$ of $\mathbb F_{q}(\cX)$ is close, since  $h\in H(P)$ if and only if $(h,\ldots,h)\in H(P_1,\ldots,P_r)$. Therefore, $i$ is a non-gap of $\mathbb F_{q}(\cX)$ if and only if $(h,\ldots,h)$ is in the Weierstrass gap set of $\{P_1,\ldots,P_r\}$; see
\cite[Proposition 2.3]{Michel}. In terms of the gap sequence at $P$, Matthews and Michel proved a bound on the minimum distance $d$ of $C_\Omega(D,G)$, namely if $G=(k+(k+t)-1)P$ where $k,\ldots,k+t\in G(P)$ and $t\geq 0$ then the Matthews-Michel bound is
\begin{equation}
\label{mm}
d\geq 2g-2+r(t+1),
\end{equation}
see \cite[Theorem 3.5]{Michel}.

Our results concern differential codes arising from a degree $3$ place on the Hermitian curve $\HC$ defined over $\mathbb F_{q^2}$.
The proofs use several geometric and combinatorial properties of $\HC$ that we quote now, the references are \cite{hi} and \cite{HP}. In the projective plane $PG(2,\mathbb F_{q^2})$ equipped with homogeneous coordinates $(X,Y,Z)$, a canonical form of $\HC$ is $X^{q+1}-Y^qZ-YZ^q=0$ so that $\HC={\bf{v}}(X^{q+1}-Y^qZ-YZ^q)$.
Every degree one place of the function field $\mathbb F_{q^2}(\HC)$ of $\HC$ corresponds to a point of $\HC$ in $PG(2,\mathbb F_{q^2})$, and this holds true for the degree one places of the constant field extension
$\mathbb F_{q^{2k}}(\HC)$ which correspond to the points of $\HC$ in $PG(2,\mathbb F_{q^{2k}})$. Moreover,
a place $P$  of degree $r>1$ of $\mathbb F_{q^2}(\HC)$ is represented by a divisor $P_1+P_2+\ldots +P_r$ of the constant field extension $\mathbb F_{q^{2r}}(\HC)$ where $P_i$ are degree one places of $\mathbb F_{q^{2r}}(\HC)$ with $P_i=\Fr^i(P_1)$ for $i=0,1,\ldots,r-1$.
Furthermore,
\[|\HC(\mathbb F_{q^2})|=|\HC(\mathbb F_{q^4})|=q^3+1, |\HC(\mathbb F_{q^6})|=q^6+1+q^4(q-1).\]
A line $l$ of $PG(2, \mathbb F_{q^2})$ is either a tangent to $\HC$ at an ${\mathbb{F}}_{q^2}$-rational point of $\HC$ or it meets $\HC$ at $q+1$ distinct ${\mathbb{F}}_{q^2}$-rational points. In terms of intersection divisors,
see \cite[Section 6.2]{HiKoToBook},
$$ I(\HC,l)=
\begin{cases}
(q+1)Q, \quad Q\in  \HC({\mathbb{F}}_{q^2});\\
\sum_{i=1}^{q+1} Q_i,\,\quad Q_i\in \HC({\mathbb{F}}_{q^2}),\,\, Q_i\neq Q_j,\,\,\, 1\le i<j\leq n.
\end{cases}
$$
Through every point $V\in PG(2,{\mathbb{F}}_{q^2})$ not in $\HC({\mathbb{F}}_{q^2})$ there are $q^2-q+1$ secants and $q+1$ tangents to $\HC$. The corresponding $q+1$ tangency points are the common points of $\HC$ with the polar line of $V$ relative to the unitary polarity associated to $\HC$. Let $V=(1:0:0)$. Then the line $l_\infty$ of equation $Z=0$ is tangent at $P_\infty=(0:1:0)$ while another line through $V$ with equation $Y-cZ=0$ is either a tangent or a secant according as $c^q+c$ is $0$ or not. This gives rise to the polynomial
\begin{equation}
\label{eqR}
R(X,Y)=X\prod_{c\in \mathbb F_{q^2},\, c^q+c\neq 0} (Y-c)
\end{equation}
of degree $q^2-q+1$. By \cite[Theorem 6.42]{HiKoToBook}, $$\Div(R(x,y))_\infty=(q^2-q+1)(q+1)P_\infty=(q^3-1)P_\infty.$$

Assume from now on that
\begin{equation}
 \label{defD} D=\sum\limits_{Q\in \HC({\mathbb F_{q^2}})\setminus \{P_\infty\}} Q.
 \end{equation}
 Proposition \ref{lm:nomequiv} below gives an explicit description of a (monomial) equivalence between the codes $C_\Omega(D,G)$ and $C_L(D,(q^3+q^2-q-2)P_\infty-G)$ constructed on $\HC$. It may be noted that this is related to the equivalence  $C_L(D,G)=C_\Omega(D,K+D-G)$ for a canonical divisor $K$, mentioned in \cite[Section III]{hop}.

The proof of  Proposition \ref{lm:nomequiv} relies on the following lemma where $\mathbb F_{q^2}(\HC)=\mathbb F_{q^2}(x,y)$ with $x^{q+1}-y^q-y=0$, and $x$ is separable function.
\begin{lemma} \label{lm:spaces}
For any divisor $E$ of\, ${\mathbb{F}}_{q^2}(\HC)$,
\begin{enumerate}
\item[\rm(i)] $\Omega(E)=dx \, \mathscr L(-E+\Div(dx))$,
\item[\rm(ii)] $ \mathscr L(D+\Div(dx)+E)=R(x,y)^{-1}\mathscr L((q^3+q^2-q-2)P_\infty +E)$.
\end{enumerate}
\end{lemma}
\begin{proof} Obviously,
$\Div(fdx)=\Div(f)+\Div(dx)\succeq E$ if and only if $\Div(f)\succeq E-\Div(dx)$, which proves (i). To
show (ii), notice that the zeros of $R(x,y)$ are the points in $\HC({\mathbb{F}}_{q^2})$ each with multiplicity one.
{}From \cite[Theorem 6.42]{HiKoToBook}, $\Div (R(x,y))=D+P_\infty-\deg R (q+1)P_\infty=D-q^3P_\infty$. Since $\Div(dx)=(2g-2)P_\infty=(q^2-q-2)P_\infty$, this gives
\[\mathscr L((q^3+q^2-q-2)P_\infty +E)=\mathscr L(D-\Div(R(x,y))+\Div(dx)+E).\]
Thus,  $f\in \mathscr L((q^3+q^2-q-2)P_\infty +E)$ and $f\in R(x,y)^{-1}\mathscr L(D-\Div(dx)+E)$ are equivalent conditions.
\end{proof}

\begin{proposition} \label{lm:nomequiv} The codes
$C_\Omega(D,G)$ and $C_L(D,(q^3+q^2-q-2)P_\infty
-G)$ are monomially equivalent.
\end{proposition}
\begin{proof} By Lemma \ref{lm:spaces}, every differential in $C_\Omega(D,G)$ can be written as $hdx$ with
$h\in \mathscr L(D-G+\Div(dx))=R(x,y)^{-1}\mathscr L((q^3+q^2-q-2)P_\infty-G)$. Let $f=gR(x,y)\in \mathscr L((q^3+q^2-q-2)P_\infty-G)$. Then $f\in{\mathbb{F}}_{q^2}[x,y]$ with $x^{q+1}-y^q-y=0.$ Also, $P_\infty$ is not a pole of $gdx$. Hence $\mathrm{res}_{P_\infty}(gdx)=0$.
Take a point $S\in \HC({\mathbb{F}}_{q^2})$ other than $P_\infty$. Then  $S=(a,b,1)$ with $b^q+ b=a^{q+1}$. Also,  $t=x-a$ is a local parameter at $S$, and the local expansion of $y$ at $S$ is
$y(t)=b+ta^q+t^{q+1}[\ldots]$. Therefore $f(a+t,y(t))=f(a,b)+t[\ldots]$ while $R(a,b)=0$ and
$R(a+t,y(t))=ut+t^2[\ldots]$ with nonzero $u$ given by
$$u=
\begin{cases}
\prod\limits_{c\in \mathbb F_{q^2},\, c^q+c\neq 0} (b-c),\quad for \,\,a=0.\\
a^{q+1}\prod\limits_{c\in \mathbb F_{q^2},\, c^q+c\neq 0,\,c\neq b} (b-c),\quad for \,\,a\neq 0.
\end{cases}
$$
Thus,
\[g(a+t,y(t))=R(a+t,y(t))^{-1}f(a+t,y(t))=u^{-1}f(a,b)t^{-1}+\cdots,\]
whence
\[\mathrm{res}_S(gdx)=\mathrm{res}_t(u^{-1}f(a,b)t^{-1}+\cdots)=u^{-1}f(S).\]
which shows the monomial equivalence between the codes $C_\Omega(D,G)$ and
$C_L(D,(q^3+q^2-q-2)P_\infty -G)$
\end{proof}
The group $\Aut(\HC)$ of all automorphisms of $\HC$ is defined over $\mathbb F_{q^2}$ and it is a projective group of $PG(2,{\mathbb{F}}_{q^2})$ isomorphic to the  projective unitary group $PGU(3,q)$. Furthermore, $\Aut(\HC)$ acts doubly transitively on $\HC({\mathbb{F}}_{q^2})$, transitively on the points of $PG(2,{\mathbb{F}}_{q^2})$ not in $\HC({\mathbb{F}}_{q^2})$, as well as on the points in $\HC(\mathbb F_{q^6}) \setminus \HC(\mathbb F_{q^2})$, and also on the set of all triangles in $\HC(\mathbb F_{q^6}) \setminus \HC(\mathbb F_{q^2})$ which are invariant under the action of the Frobenius map. The latter property shows that the geometry of degree $3$ places of $\mathbb F_{q^2}(\HC)$ is independent on the choice of $P$.
Write $P=P_1+P_2+P_3$ with $P_i\in \HC(\mathbb F_{q^6}) \setminus \HC(\mathbb F_{q^2})$ and fix a projective frame $(X_1,X_2,X_0)$ whose vertices are the points $P_i$. For a suitable choice of the unity point $U_0\in \HC(\mathbb F_{q^2})$, the equation of $\HC$ becomes $$X_1X_2^q+X_2X_0^q+X_0^qX_1=0,$$ see
\cite[Proposition 4.6]{CTK1} where the non-singular  matrix $M$ realizing the change of coordinates $(X,Y,Z)\to (X_1,X_2,X_0)$ is given explicitly. In doing so, every $f\in \HC({\mathbb{F}}_{q^2})$ will have an equation in $(X_1,X_2,X_0)$. In other words, the linear map $\mu$ of $\HC(\mathbb F_{q^6})$ associated to $M$ takes $\HC(\mathbb F_{q^2})$ to a subfield  $\HC(\mathbb F_{q^6})$ which is isomorphic to (but distinct from) $\HC(\mathbb F_{q^2})$.

For $i=0,1,2\,\pmod 3$, the tangent to $\HC$ at $P_i$ is the line $l_i=P_iP_{i+1}$ of equation $X_{i+1}=0$.   Therefore \begin{equation}
\label{eqjun312}
I(\HC\cap l_i)=qP_i+P_{i+1},\quad i=0,1,2 \pmod 3.
\end{equation}
Let $l_i={\bf{v}}(\ell_i)$. Then
\begin{align*}
\Div(\ell_1)&=qP_1+P_2-(q+1)P_\infty,\\
\Div(\ell_2)&=qP_2+P_3-(q+1)P_\infty,\\
\Div(\ell_0)&=qP_3+P_1-(q+1)P_\infty,\\
\Div(\ell_1\ell_2\ell_0)&=(q+1)P-3(q+1)P_\infty.\\
\end{align*}
Observe that $\curve{\ell_1\ell_2\ell_0}$ is defined over ${\mathbb{F}}_{q^2}$ while $l_i$ is defined over  ${\mathbb{F}}_{q^6}$.
\begin{lemma}
\label{lemtech} Let $\cC$ be any (possible singular or reducible) plane curve not containing the tangent $l_i$ to $\HC$ at $P_i$ as a component where $0\le i \le 2$. If $I(P_i,\HC\cap \cC)\leq q$, then
$$I(P_i,\HC\cap \cC)=I(P_i,l_i\cap \cC).$$
\end{lemma}
\begin{proof} We prove the assertion for $i=1$. We use affine coordinates $(X,Y)$ with $X=X_1/X_0,\,Y=X_2/X_0$ so that $\HC$ has equation $Y+X^q+XY^q=0$ and $P_1=(0,0)$. Then $X$ is a local parameter at $P_1$ and the expansion of $Y$ is $Y(X)=X^q(-1+X[\ldots])$. Furthermore, $\ell_1$ has equation $Y=0$. Let $F(X,Y)=0$ be an affine equation of $\cC$.  Then $I(P_1,\ell_1\cap \cC)=m$ if and only if $F(X,0)=c_1X^m(c_2+X[\ldots])$ with nonzero $c_1,c_2\in \bar{\mathbb{F}}_{q^2}$. Since $Y$ does not divide $F(X,Y)$ and $I(P_i,\HC\cap \cC)\leq q$ , we also have $F(X,Y(X))=d_1X^m(d_2+X[\ldots])$ with nonzero  $d_1,d_2\in \bar{\mathbb{F}}_{q^2}$. Therefore  $I(P_1,\HC\cap \cC)=m$.
\end{proof}
{}From the above discussion we have the following result
\begin{proposition} \label{lm:mreduct}
Let $m=m_1(q+1)+m_0$ with $m_1$ and $m_0$ non-negative integers such that $m_0\le q$. In ${\mathbb{F}}_{q^2}(\HC)$, take a degree $3$ place $P$ together with a degree one place $P_\infty$ ${\mathbb{F}}_{q^2}$-rational. Let
\begin{align*}
A_1&=(q^3+q^2-q-2)P_\infty-m P,\\
A_2&=(q^2-3m_1-1)(q+1)P_\infty -(P_\infty+m_0 P).
\end{align*}
Then the codes $C_L(D,A_1)$ and $C_L(D,A_2)$ are monomially equivalent.
\end{proposition}
\begin{proof}
The monomial equivalence of the two codes follows from $A_2=A_1+m_1(\ell_1\ell_2\ell_3)$ after observing that the $\mathbb{F}_{q^2}$-rational polynomial $\ell_1\ell_2\ell_0$ has neither zeros nor poles in $\supp D$.
\end{proof}
\begin{rem}
\label{rem1}
\em{By Propositions \ref{lm:nomequiv} and \ref{lm:mreduct}, the differential code $C_\Omega(D,mP)$ and the functional code $C_L(D,(q^2-3m_1-1)(q+1)P_\infty -(P_\infty+m_0 P))$ are monomially equivalent.
They have length $q^3$, dimension $q^3+\frac{1}{2}(q^2-q-2)-3m$ and designed minimum distance
\begin{equation}
\label{dmd}\delta=3m-q^2+q+2.
\end{equation}
In particular, $3m\ge q^2-q-2\geq 0$ holds.}
\end{rem}
\begin{rem}
\label{rem1bis}
\em{Propositions \ref{lm:nomequiv} shows that if $m_0=0$ then $C_L(D,A_2)$ is $C_L(D,tP_\infty)$ with $t=(q^2-3m_1-1)(q+1)$. For such particular codes, the minimum distance problem has been solved in \cite{stichbis,YK}. Therefore we may limit ourselves to the case where $m=m_1(q+1)+m_0$ with $m_0>0$.}
\end{rem}

\section{The Weierstrass gap sequence of places of higher degree}
As we have pointed out in the Introduction, in the study of differential codes $C_\Omega(D,G)$ where $\supp(G)$ consists of just one place $P$, possibly of degree $r>1$, a key issue is to determine the gap sequence at $P$. In the case where $P$ has degree one, this essentially requires to determine the Weierstrass semigroup at $P$ and the relative computations can generally be carried out using methods from classical algebraic geometry. For instance, for the Hermitian function field $\mathbb{F}_{q^2}(\HC)$, the Weierstrass semigroup is as simple as possible being generated by $q$ and $q+1$. The analog question for places of degree $r>1$ is still open even for $\mathbb{F}_{q^2}(\HC)$, apart from some smallest values of $q$ namely $q\leq 9$ where the computations were carried out by using the MAGMA; see \cite{Michel}.

In this section we determine the gap sequence of $\mathbb{F}_{q^2}(\HC)$ at any place $P$ of degree $3$, see Theorem \ref{th:HP}. In turns out that the smallest non-gap is $q-2$, and we first explain why this occur.

There exists $\alpha\in {\rm{Aut}}(\HC)$ of order $3$ which has no fixed point off $\HC(\mathbb{F}_{q^2})$ and acts on $\{P_1,P_2,P_3\}$ as a $3$-cycle. The quotient curve $\cC=\HC/\langle \alpha \rangle$ is a $\mathbb{F}_{q^2}$-maximal curve. Furthermore, the place of $\bar{P}$ of $\mathbb{F}_{q^2}(\cC)$ lying under $P$ is unramified and the smallest non-gap at $\bar{P}$ is $q-2$. Take  $f\in \mathbb{F}_{q^2}(\cC)$ such that $\Div(f)_{\infty}=(q-2)\bar{P}$. Then $f$ can also be viewed as an element of $\mathbb{F}_{q^2}(\HC)$ and  $\Div(f)_{\infty}=(q-2)P$ remains true in $\mathbb{F}_{q^2}(\HC)$. Viceversa, if $i<q-2$ is a non-gap at $P$, let $f\in \mathbb{F}_{q^2}(\HC)$ with $\Div(f)_{\infty}=iP$ and $f^{\alpha}=f$. The latter property implies that $f\in \mathbb{F}_{q^2}(\cC)$ with $\Div(f)_{\infty}=i\bar{P}$. But this is impossible since $q-2$ is the smallest non gap at $\bar{P}$.
\begin{theorem} \label{th:HP}
For any degree $3$ place $P$ of $\mathbb{F}_{q^2}(\HC)$, the Weierstrass gap sequence at $P$ is
\begin{equation}
\label{eq13jun12}
G(P)=\{u(q+1)-v \mid 0\leq v\leq q, 0<3u\leq v\}.
\end{equation}
\end{theorem}
\begin{proof} For two integers $u,v$ with $0\le v\le q,\,0<3u\le v$, let $m=u(q+1)-v$. First we construct the complete linear series $|m(P_1+P_2+P_3)|$ using \cite[Theorem 6.52]{HiKoToBook}. From (\ref{eqjun312}), we have $\sum_{i=0}^2\,I(P_i,\HC\cap l_i)=(q+1)(P_1+P_2+P_3)$. This shows that the curve ${\bf{v}}((\ell_1\ell_2\ell_0)^u)$ of degree $3u$ is an adjoint of the divisor $m(P_1+P_2+P_3)$. Therefore, up to the fixed divisor $v(P_1+P_2+P_3)$, the complete linear series $|m(P_1+P_2+P_3)|$ consists of the divisors cut out by the adjoint curves $\Phi$ of degree $3u$ for which $I(P_i,\HC\cap \Phi)\geq v$ for $i=0,1,2$. Reformulating this in terms of Riemann-Roch spaces; see \cite[Section 6.4]{HiKoToBook}, gives

\[\mathscr L(mP)=\left\{\frac{f}{(\ell_1\ell_2\ell_3)^u} \mid f\in \mathbb{F}_{q^2}[X,Y],\deg f\le 3u, v_{P_i}(f) \geq v\right\}\cup\{0\}.\]
Since $v\leq q$ and the tangent line at $P_i$ is $\curve{\ell_i}$, this together  with Lemma \ref{lemtech} yield $I(\ell_i \cap \curve{f},P_i)\geq v$. Moreover, $P_{i+1}\in \curve{f}\cap \curve{\ell_i}$. Therefore, counted with multiplicity, $\curve{\ell_i}$ and $\curve{f}$ have at least $v+1$ common points. If $\deg \curve{f}= 3u\leq v$ then B\'ezout's theorem, see \cite[Theorem 3.14]{HiKoToBook}, yields  $\ell_i\mid f$. This holds for $i=0,1,2$. Thus, $\ell_1\ell_2\ell_3 \mid f$. Hence $f/(\ell_1\ell_2\ell_3)^u)=g/(\ell_1\ell_2\ell_3)^{u-1}$ with $\deg g\leq 3(u-1)$. This yields that $L(mP)\subseteq L((m+1)P)$. Therefore, the right hand side in (\ref{eq13jun12} is indeed in $G(P)$.

Viceversa, assume that $0\leq v\leq q$ and $3u>v$. Let $w=\ell_1^{2u-v} \ell_2^{v-u} \ell_3^{-u}$. Then $\Div(w)=m_1P_1+m_2P_2+m_3P_3$, where
\begin{align*}
m_1&=(2u-v)q-u,\\
m_2&=(v-u)q+2u-v,\\
m_3&=-uq-u+v.
\end{align*}
Obviously, $m_3=-m$. Also, $m_2\leq m_3$ is equivalent to
$vq\leq 2v-3u<2v.$
Since $q\ge 2$, this yields $v=0$ and $0\leq -3u$, a contradiction. Now, assume $m_1\leq m_3$. Then
$(3u-v)q\leq v \leq q,$
which implies $3u-v\leq 1$. As $3u>v$, this yields $3u=v+1$ and $v=q$ whence $m=\frac{1}{3}(q^2-q+1)$ follows. Thus,  $\deg (mP)=3m>2g-1$, where $g=\frac{1}{2}\,q(q-1)$ is the genus of $\HC$. From \cite[Proposition 2.1]{Michel}, $m$ is not in $G(P)$.

We are left with the case where $m_1,m_2>m_3=-m$. For $w\in \mathbb{F}_{q^6}(\HC)$, let ${\rm{Tr}}(w)=
w+\Fr(w)+\Fr^2(w)$. Obviously ${\rm{Tr}}(w)\in \mathbb{F}_{q^2}(\HC)$. Furthermore,
\begin{align*}
v_{P_i}({\rm{Tr}}(w))&\leq\min \{v_{P_i}(w),v_{P_i}(\Fr(w)),v_{P_i}(\Fr^2(w)))\} \\
&=\min \{v_{P_1}(w),v_{P_2}(w),v_{P_3}(w)\} \\
&=\min \{m_1,m_2,m_3\}=-m
\end{align*}
for $i=0,1,2$. As the minimum is unique by assumption, the equality holds. Therefore $m$ is not in $G(P)$.
\end{proof}
As a corollary we have the following result.
\begin{corollary}
\label{corgap}
The maximal consecutive gap sequences in $G(P)$ are $(u-1)q+u, \ldots, u(q-2)$, where $u$ is an integer satisfying $0< 3u \leq q$.
\end{corollary}

\section{On the Matthews-Michel bound for AG-codes from Hermitian curves}
Corollary \ref{corgap} allows us to compute explicitly the Matthews-Michel bound (\ref{mm}) on the minimum distance for  any one-point differential code $C_\Omega(D,mP)$ constructed on $\HC$ where $P$ is a degree $3$ place and $D$ is defined by (\ref{defD}). Indeed, from Corollary \ref{corgap} we can read out the consecutive gap sequences in $G(mP)$, the longest are $\alpha=(u-1)q+u, \ldots, \alpha+t=u(q-2)$ when
\[m=2\alpha+t-1=m_1(q+1)+m_0, \hspace{5mm} m_1=2u-2, \hspace{5mm} m_0=q+1-3u.\]
For such a sequence, the Matthews-Michel bound is $(q-2)(6u-q-1)$ and it gives an improvement on the designed minimum distance by $3(t+1)=3(q+1-3u)=3m_0$. It should be noted that the improvement is nontrivial when $m_1=2u-2$ satisfies
the condition $q-4\leq 3m_1 \leq 2(q-3).$ From the above discussion we have the following result.

\begin{theorem} \label{pr:MathewsMichel}
Let $\HC$ be the Hermitian curve over $\mathbb{F}_{q^2}$. Define $P$ to be a degree $3$ place in $\HC(\mathbb{F}_{q^2})$ and $D$ to be the divisor defined by (\ref{defD}).
Let $u$ be an integer with $q+1\leq 6u\leq 2(q+1)$. Let $m=(2u-1)q-u-1=m_1(q+1)+m_0$ with $0\le m_0\le q$. Then the minimum distance of the differential code $C_\Omega(D,mP)$ is at least $$\delta+3(q+1-3u)=\delta+3m_0.$$ where $\delta$ is the designed minimum distance of the code given in (\ref{dmd}).
\end{theorem}

\section{Improvements on the Matthews-Michel bound} \label{sectionnoether}
Remark \ref{rem1} tells us that the parameters of the differential code $C_\Omega(D,mP)$ may be investigated using
the functional code
\begin{equation}
\label{eq44jun12}
C_L(D,(q^2-3m_1-1)(q+1)P_\infty -(P_\infty+m_0 (P_1+P_2+P_3)).
 \end{equation} The advantage is that more geometry can be exploited, and we will do it with an approach based on the  Noether ``AF+BG'' theorem, see \cite[Theorem 4.66]{HiKoToBook}. For our particular need, we state this theorem in the following form.
\begin{lemma}
\label{lemnother}
Let $\cF=\curve{F}$ and $\cC=\curve{C}$ be any two (possible singular or reducible) curves defined over  $\bar{\mathbb{F}}_{q^2}$ such that $I(\cF\cap \HC) \succeq I(\cC\cap \HC).$ Then there exist $A,B\in \bar{\mathbb{F}}_{q^2}[X,Y]$ with $F=AC+BH$. If both $\cF$ and $\cC$ are defined over ${\mathbb{F}}_{q^2}$,  then $A,B$ can be chosen in ${\mathbb{F}}_{q^2}[X,Y]$.
\end{lemma}
Here, we take $C(X,Y)$ to be the polynomial whose evaluation in $D$ gives a codeword with minimum distance in (\ref{eq44jun12}). The curve $\cC=\curve{C}$ has degree $q^2-3m_1-1$ and  $I(\HC\cap \cC)\succeq
P_\infty+m_0(P_1+P_2+P_3)$. In fact, the complete linear series $|(q^2-3m_1-1)(q+1)P_\infty -(P_\infty+m_0 (P_1+P_2+P_3))|$ is cut out, up to fixed divisor $P_\infty+m_0 (P_1+P_2+P_3)$, by the (adjoint) curves $\cA$ of degree $q^2-3m_1-1$  satisfying the condition $I(\HC\cap \cA) \succeq P_\infty+m_0 (P_1+P_2+P_3)$. In terms of $\cC$, the minimum distance $d$ of (\ref{eq44jun12}) is equal to $q^3-N$ where N is the number of points of $\HC({\mathbb{F}}_{q^2})\setminus \{P_\infty\}$ which are also points of $\cC$.

Let $r_0$ be the non-negative integer satisfying $I(P_i,\curve{C} \cap \HC)= m_0+r_0$. From B\'ezout's theorem, see \cite[Theorem 3.14]{HiKoToBook}, $$(q^2-3m-1)(q+1)=\deg \cC \deg \HC\geq (q^3-d)+3(m_0+r_0)$$ whence $d\geq \delta+3r_0$ with $\delta$ being the designed minimum distance, see (\ref{dmd}) in Remark \ref{rem1}.
\begin{lemma} \label{lm:m0+r0q1}
If $m_0+r_0\geq q+1$ then $m_0+r_0=q+1$ and the minimum distance is $d=\delta+3(q+1-m_0)$ where $\delta$ is the designed minimum distance given in (\ref{dmd}).
\end{lemma}
\begin{proof}
Let $C^*(X,Y)=\ell_1\ell_2\ell_3 X (Y-c_1)\cdots(Y-c_k)$ for $k+4=q^2-3m_1-1$ with $c_i^q+c_i\neq 0$. Obviously, $C^*(x,y) \in \mathscr L(A_2)$. Also, $I(\curve{C^*} \cap \HC)=P_\infty +(q+1)(P_1+P_2+P_3) +B$ where $B$ is the sum of $q+(q+1)(q^2-3m_1-5)$ points in $\HC(\mathbb{F}_{q^2})$. The weight of the corresponding codeword $\mathbf{c}^*$ is
\begin{equation}
\label{eq5giun12}
d^*=q^3-\deg B=3m_1(q+1)-q^2+4q+5=\delta +3(q+1-m_0).
\end{equation}
Now, $d^*\geq d\geq \delta+3r_0$ together with $m_0+r_0\geq q+1$ yield $r_0=q+1-m_0$ whence  $d=d^*$.
\end{proof}
\begin{rem}
\label{rem2} \em{From (\ref{eq5giun12}),  a lower bound for the minimum distance of (\ref{eq44jun12}) is $\delta+3(q+1-m_0)$ with $\delta$ designed minimum distance given in (\ref{dmd}).}
\end{rem}
As we have pointed out, there are precisely $d$ $\mathbb{F}_{q^2}$-rational points in $\HC$ not on $\curve{C}$. Let $E_0$ be the sum of the $\mathbb{F}_{q^2}$-rational points in $\supp I(\curve{C}\cap \HC)$.
Then
\[I(\curve{C} \cap \HC)=E_0+E+(m_0+r_0)P,\]
where $r_0\geq 0$ and $E$ is an effective divisor defined over $\mathbb{F}_{q^2}$. The minimum distance $d$ satisfies
\begin{equation}
\label{eq77}
d=\delta+\deg E+3r_0,
\end{equation}
with designed minimum distance given in (\ref{dmd}).

For a given integer $1\le \alpha\le q$, let $|U|$ be the complete linear series cut out on $\HC$ by all plane curves of degree $\alpha$. Then $||U|-|E||$ is a complete linear series consisting of all intersection divisors $I(\cF\cap\HC)$
with $\cF$ ranging over all plane curves of degree $\alpha$; see \cite[Theorem 6.40]{HiKoToBook}. If $\dim (||U|-|E||)\geq 0$ then $||U|-|E||$ contains a divisor cut out by a curve defined over  $\mathbb{F}_{q^2}$, as $E$ itself is defined over $\mathbb{F}_{q^2}$. Furthermore, since $\dim U=
\frac{1}{2}\alpha(\alpha+3)$, \cite[Corollary 6.27]{HiKoToBook} gives $\dim(|U|-|E|)\geq  \frac{1}{2}\alpha(\alpha+3)-\deg E$. If we take the minimum value of $\alpha$ for which \begin{equation}
\label{eq88}  \deg E \leq \frac{1}{2}\,\alpha(\alpha+3),
\end{equation}
then $||U|-|E||\neq \emptyset$.
In terms of Riemann-Roch spaces,  the $\bar{\mathbb{F}}_q$-linear space
\[\mathbf T_{\alpha}=\{ T \in \bar{\mathbb{F}}_q[X,Y] \mid \deg T\leq \alpha, I(\curve{T} \cap \HC)\succeq E \},\]
has
\begin{equation*}
\dim \mathbf{T}_{\alpha}\geq \frac{1}{2}(\alpha+1)(\alpha+2)-\deg E.
\end{equation*}
and if $\alpha$ is chosen according to (\ref{eq88}) then $\mathbf{T}_\alpha$ is nontrivial.
Noether ``AF+BG'' theorem gives the following result.
\begin{lemma} \label{lm:Noether}
Assume $m_0+r_0\leq q$. Then for any nonzero $T\in \mathbf{T}_\alpha$ there are polynomials $A,B\in \bar{\mathbb{F}}_q[X,Y]$ such that
\begin{equation} \label{eq:Noether}
T\ell_1\ell_2\ell_3 R=AC+BH.
\end{equation}
If $T$ is defined over $\mathbb{F}_{q^2}$ then so are $A,B$, as well.
\end{lemma}
\begin{proof}
 {}From the definition of $T$, $$I(Q,\curve{T\ell_1\ell_2\ell_3 R} \cap \HC) \geq I(Q,\curve{C} \cap \HC)$$ for all points $Q\in PG(2,\bar{\mathbb{F}}_{q^2})$ of $\HC$. Therefore, Lemma \ref{lemnother} applies.
\end{proof}
{}From now on, whenever a fixed nonzero $T\in \mathbf{T}_\alpha$ is given, then $A,B$ will denote a polynomials satisfying \eqref{eq:Noether}. Comparing the degrees in \eqref{eq:Noether} gives
\begin{equation} \label{eq:degA}
\deg A=3m_1-q+5+\alpha.
\end{equation}
\begin{lemma} \label{lm:PionAB}
Assume $m_0+r_0\leq q$ and let $0\neq T\in \mathbf{T}_\alpha$. Then $P_1,P_2,P_3 \in \curve{A} \cap \curve{B}$.
\end{lemma}
\begin{proof}
As $I(P_i,\curve{\ell_1\ell_2\ell_3}\cap \HC)=q+1$ and $I(P_i,\curve{R}\cap \HC)=0$, we have
\[I(P_i,\curve{A}\cap \HC)+I(P_i,\curve{C}\cap \HC)=I(P_i,\curve{T}\cap \HC)+q+1+0.\]
This implies $I(P_i,\curve{A}\cap \HC)\geq q+1-m_0-r_0$, and $P_i\in \curve{A}$. To prove $P_i\in \curve{B}$, observe  first that if $\ell_{i-1} \mid A$ then $\ell_{i-1} \mid B$ and $P_i\in \curve{B}$. Assume $\ell_{i-1} \nmid A$. From $P_i=\ell_{i-1}\cap \ell_i$,
\[I(P_i,\curve{\ell_{i-1}}\cap \curve{A})+I(P_i,\curve{\ell_{i-1}}\cap \curve{C})\geq 2.\]
Therefore $I(P_i,\curve{\ell_{i-1}}\cap \curve{B})\geq 1$ follows from $I(P_i,\curve{\ell_{i-1}}\cap \HC)=1$.
\end{proof}

\begin{lemma} \label{lm:alpha+r0}
Assume $m_0+r_0\leq q$, and suppose that there is a nonzero $T\in \mathbf{T}_\alpha$ such that $\ell_i\nmid A$. Then, $\alpha+r_0\geq 2q-3m_1-m_0-3$.
\end{lemma}
\begin{proof}
Since $m_0+r_0\leq q$ and $\curve{\ell_i}$ is the tangent line to $\HC$ at $P_i$, $$I(P_i,\curve{C} \cap \curve{\ell_i})=I(P_i,\curve{C}\cap \HC)=m_0+r_0.$$ Moreover,
\[\deg A-1 +m_0+r_0\geq I(P_i,\curve{A} \cap \curve{\ell_i})+I(P_i,\curve{C} \cap \curve{\ell_i}),\]
and
\[I(P_i,\curve{B} \cap \curve{\ell_i})+I(P_i,\HC \cap \curve{\ell_i})\geq 1+q.\]
This implies $\deg A-1 +m_0+r_0 \geq 1+q$. The result follows from \eqref{eq:degA}.
\end{proof}

\begin{lemma} \label{lm:elldivA}
Assume $m_0+r_0\leq q$, $\mathbf{T}_\alpha \neq 0$, and $\ell_1\ell_2\ell_3 \mid A$ for all $0\neq T\in \mathbf{T}_\alpha$. Then $\alpha\geq m_0+r_0+1$.
\end{lemma}
\begin{proof}
If $\ell_1\ell_2\ell_3 \mid A$ then $\ell_1\ell_2\ell_3 \mid B$ and $TR=A'C+B'H$ with polynomials $A',B'$. Take $\alpha$ to be the least integer with $\mathbf{T}_\alpha \neq 0$, see (\ref{eq88}). Since $\supp(E)\cap \supp I(\HC\cap\curve{\ell_1\ell_2\ell_3})=\emptyset$, we have $\ell_i \nmid T$. The equation
\[I(P_i,\curve{T} \cap \HC) + \overbrace{I(P_i,\curve{R} \cap \HC)}^{=0} = I(P_i,\curve{A'} \cap \HC) + \overbrace{I(P_i,\curve{C} \cap \HC)}^{=m_0+r_0}\]
implies $m_0+r_0\leq I(P_i,\curve{T} \cap \HC) = I(P_i,\curve{T} \cap \curve{\ell_i})$. Hence, counted with multiplicity, the line $\ell_i=0$ has at least $m_0+r_0+1$ points in common with $T=0$. This implies $\alpha\geq \deg T\geq m_0+r_0+1$.
\end{proof}
\begin{lemma}  \label{lm:coefftrick}
Assume $2\leq m_0+r_0\leq q$ and let $T \in \mathbf{T}_\alpha$ be a nonzero polynomial such that $P_i\in \curve{T}$ and $\ell_i \nmid A$ for some $i\in \{1,2,3\}$. Then, $I(P_i,\curve{\ell_i} \cap \curve{B})\geq 2$.
\end{lemma}
\begin{proof}
We prove the assertion for $i=1$. Take $P_1P_2P_3$ to be the fundamental triangle of a homogeneous coordinate system $(X,Y,Z)$, and use inhomogeneous coordinates where $Z=0$ the infinite line, and $P_1$ is the origin.  Then
\begin{itemize}
\item[(a)] $T(0,0)=0,\, R(0,0)\neq 0,\, \ell_1\ell_2\ell_3=XY;$
\item[(b)] $A(X,Y)=Y(a_1+\ldots)+X^{q+1-(m_0+r_0)}(a_2+\ldots)$;
\item[(c)] $C(X,Y)=c_1Y+c_2X^{m_0+r_0}+\ldots;$
\item[(d)] $B(X,Y)=b_0+b_1X+b_2Y+\ldots,\, H(X,Y)=Y+X^q+XY^{q+1}$.
\end{itemize}
By Lemma \ref{lm:PionAB}, $b_0=0$. Observe that the polynomials $T\ell_1\ell_2\ell_3R$ and $AC$ contain no term $XY$. From $BH=T\ell_1\ell_2\ell_3R-AC$, the coefficient of $XY$ in the polynomial $BH$ must vanish. This yields $b_1=0$. Therefore,
\[2\leq I(P_1,\curve{B}\cap \curve{Y})=I(P_1,\curve{B}\cap \curve{\ell_1})=I(P_1,\curve{T} \cap \HC),\]
whence the assertion follows.
\end{proof}
\begin{lemma} \label{lm:multP1onT}
Assume $m_0+r_0\leq q$ and let $T \in \mathbf{T}_\alpha$ be a nonzero polynomial such that $\ell_i \mid A$ for some $i\in \{1,2,3\}$. Then, either $\ell_i\mid T$, or $I(P_i,\curve{\ell_i} \cap \curve{T})\geq m_0+r_0-1$.
\end{lemma}
\begin{proof}
We prove the assertion for $i=1$. If $\ell_1\mid A$ then $\ell_1\mid B$ and $T\ell_2\ell_3R=A'C+B'H$ for some polynomials $A',B'$. On the one hand, $$I(P_1,\curve{T\ell_2\ell_3R}\cap \HC)=I(P_1,\curve{T}\cap\HC)+0+1+0.$$ On the other hand, $I(P_1,\curve{A'C} \cap \HC)=I(P_1,\curve{A'} \cap \HC)+m_0+r_0$. Thus,
\[I(P_1,\curve{T}\cap \HC)\geq m_0+r_0-1,\]
whence the assertion follows.
\end{proof}
We are in a position to prove our main result.
\begin{theorem} \label{th:main}
Let $m$ be an integer such that $q^2-q-2\leq 3m \leq 2q^2-q-2$ and $q+1 \nmid m$. Let $d$ and $\delta$ be the minimum distance and the designed minimum distance of the differential code $C_\Omega(D,mP)$, respectively. Write $m=m_1(q+1)+m_0$ with $0< m_0\le q$. Assume that
\begin{equation}
\label{hypK}
K=2q-3m_1-m_0-4\geq 0.
\end{equation}
Then one of the following holds:
\begin{enumerate}
\item[\rm(i)] $d=\delta+3(q+1-m_0)$.
\item[\rm(ii)]$d\geq \delta +\frac{1}{2}\,(m_0+1)(m_0+2)$.
\item[\rm(iii)] $d\geq \delta +3K$ and if $d=\delta+3K$ then $m_0\leq 2$.
\end{enumerate}
\end{theorem}
\begin{proof}
We continue to work on the equivalent functional code (\ref{eq44jun12}) and use the above notation.
If $m_0+r_0\geq q+1$ then (i) holds by Lemma \ref{lm:m0+r0q1}. Assume $m_0+r_0\leq q$. According to the discussion made before Lemma \ref{lm:Noether}, we may choose $\alpha$ such that
\[\frac{\alpha(\alpha+3)}{2}\geq \deg E \geq \frac{\alpha(\alpha+1)}{2}.\]
 $\mathbf{T}_\alpha \neq 0$. If for all nonzero $T\in \mathbf{T}_\alpha$, $\ell_1\ell_2\ell_3 \mid A$ then $\alpha\geq m_0+1$ by Lemma \ref{lm:elldivA}, and case (ii) occurs by (\ref{eq88}).

 Therefore, we may suppose the existence of $T\in \mathbf{T}_\alpha\setminus \{0\}$ such that $\ell_1 \nmid T$. By Lemma \ref{lm:alpha+r0}, $\alpha+r_0\geq K+1$ and
\begin{align*}
\deg E+3r_0&\geq \frac{\alpha(\alpha+1)}{2}+3r_0 \\
&\geq \frac{(K+1-r_0)(K+2-r_0)}{2}+3r_0 \\
&= \frac{(K-r_0-\frac{3}{2})^2-\frac{1}{4}}{2}+3K \\
&\geq 3K.
\end{align*}
This proves $d\geq \delta+3K$, and also shows that $d=\delta+3K$ if and only if equality occurs everywhere in the last computation. Therefore $$K-r_0\in \{1,2\},\quad \alpha=K+1-r_0\in \{2,3\},\quad  \deg E =\frac{1}{2}\,\alpha(\alpha+1)\in \{3,6\}.$$ It remains to show $m_0\leq 2$.

Assume $m_0\geq 3$, and define the subspace
\[\tilde{\mathbf{T}}_\alpha =\{T\in \mathbf{T}_\alpha \mid P_1\in \curve{T}\}\]
of $\mathbf{T}_\alpha$. Suppose that there is a nonzero polynomial $T\in \tilde{\mathbf{T}}_\alpha$ such that $\ell_1\nmid A$. Then Lemma \ref{lm:coefftrick} improves the inequality in Lemma \ref{lm:alpha+r0} by $1$.

Assume $\ell_1\mid A$ for all nonzero polynomials $T\in \tilde{\mathbf{T}}_\alpha$, and investigate several cases separately.
\begin{description}
\item[Case 1:] $\deg E=3$ and $I(\HC\cap r)\succeq E$ for some line $r$.
\end{description}
In this case $\alpha=2$. Define the quadratic polynomial $T$ to be the product $T=UV$, where $\deg U=\deg V=1$, $E\preceq I(\curve{U}\cap \HC)$ and $\curve{V}$ is a line through $P_1$ different from $l_1$. Then $T\in \tilde{\mathbf{T}}_2$. Since $\ell_1\nmid T$, Lemma \ref{lm:multP1onT} yields $I(P_1,\curve{T} \cap \curve{\ell_1}) \geq 2$, a contradiction.

\begin{description}
\item[Case 2:] $\deg E=3$ and there exists no line $r$ with $I(\HC\cap r)\succeq E$.
\end{description}
Let $\curve{T}$ be a non-degenerate conic such that $E+P_1+P_2 \preceq I(\curve{T}\cap \HC).$ By our assumption, the case $\ell_1\mid T$ cannot occur. Therefore, Lemma \ref{lm:multP1onT} yields $I(P_1,\curve{T} \cap \curve{\ell_1}) \geq 2$. As $P_2\in \curve{T} \cap \curve{\ell_1}$, counted with multiplicity, the line $l_1$ has $3$ common intersections with $\curve{T}$, a contradiction.

\begin{description}
\item[Case 3:] $\deg E=6$ and $I(\HC\cap \cF)\succeq E$ for some conic $\cF$.
\end{description}
Since $E$ is defined over $\mathbb{F}_{q^2}$, there exists a (possible degenerate) $\mathbb{F}_{q^2}$-rational conic $\curve{T}$ such that $E\preceq I(\curve{T} \cap \HC)$. Then $A$ is also defined over $\mathbb{F}_{q^2}$.

Assume first that $\curve{T}$ contains one of the points $P_i$, then it also contains each point $P_i$ with $i=0,1,2$.  Hence $T\in \tilde{\mathbf{T}}_2$. By our assumption $\ell_1\mid A$, and hence $\ell_1\ell_2\ell_0\mid A$. But this is impossible by  Lemma \ref{lm:elldivA}.

Therefore $P_1\not \in \curve{T}$.
Let $T^*=TU$, where $\curve{U}$ is a line through $P_1$ different from $l_1$. As $T^*\in \tilde{\mathbf{T}}_3\setminus \{0\}$, we have $\ell_1 \nmid T^*$ and hence $\ell_1\mid A^*$ by our assumption, Lemma \ref{lm:multP1onT} implies $I(P_1,\curve{T^*}\cap \curve{\ell_1}) = I(P_1,\curve{U}\cap \curve{\ell_1})\geq 2$, a contradiction.
\begin{description}
\item[Case 4:] $\deg E=6$ and there is no conic $\cF$ such $I(\HC\cap \cF)\succeq E$.
\end{description}
Since $\deg (E+P_2+P_0)=8$, there exists a (possible singular or degenerate) cubic curve $\curve{T}$ tangent to $l_1$ to $P_2$ such that $E+P_2+P_0\preceq I(\curve{T} \cap \HC)$. With this choice $l_1$ is not a component of $\curve{T}$.
In fact, if $T=\ell_1 F$ then
$$I(\HC\cap \curve{T})=I(\HC\cap l_1)+I(\HC\cap \curve{F})=qP_1+P_2+I(\HC\cap \curve{F}),$$
and this together with $I(\HC\cap \curve{T})\succeq E+P_2+P_0$ yield $I(\HC\cap \curve{F})\succeq E$. But this is a contradiction as $\deg F=2$.

Furthermore $T\in \tilde{\mathbf{T}}_3$ is a nonzero polynomial. Hence Lemma \ref{lm:multP1onT} implies $I(P_1,\curve{T}\cap \curve{\ell_1})\geq 2$. Therefore
$$\deg(I(\curve{T}\cap l_1))\geq I(P_1,\curve{T}\cap l_1)+I(P_2,\curve{T}\cap l_1)\geq 4.$$
Again a contradiction as $l_1$ is not a component of $\curve{T}$.
\end{proof}
\begin{rem}
\label{rem9jun12} \em{By hypothesis (\ref{hypK}) and Remarks \ref{rem1}, \ref{rem1bis}, Theorem \ref{th:main} applies to $m$ in the range
\begin{equation}
\label{mainappl} \frac{1}{3}\,(q-1)(q+1) \leq m \leq \frac{2}{3}\,q(q+1),\quad (q+1)\nmid m.
\end{equation}
 }
\end{rem}

\section{Examples} \label{sectionex}
First we compare our bound with the Matthews-Michel bound as stated in Theorem \ref{pr:MathewsMichel}.
It turns out that Theorem \ref{th:main} implies the Matthews-Michel bound for all possible values of $u$. Actually,  an effective improvement occurs apart from exceptional cases, namely:
\begin{enumerate}[(1)]
\item[\rm(i)] if $m_0+r_0\geq q+1$ then we have an exact value for the minimum distance of $C_\Omega(D,mP)$;
\item[\rm(ii)] if $m_0=1$ or $2$.
\end{enumerate}
In case (ii), several extra information can be obtained on the geometry of the minimum distance codeword. Using this knowledge, we were able to find with a computer aided search by MAGMA and GAP4 \cite{GAP4} that for $q=7$, the differential code $C_\Omega(D,18P)$ has a codeword of weight $d=20$, see the program code in Appendix \ref{app:source}. Therefore, the minimum distance is at most $20$, showing the sharpness of the Matthews-Michel bound for this specific case.

Next, we present a comparison of our bound with the true values of the minimum distances of Hermitian $1$-point codes; see \cite{stichbis,YK} and \cite[Table 1]{XC}. The parameters of the code $C_\Omega(D,mP)$ can be compared with the parameters of the $1$-point differential code $C_\Omega(D,3mP_\infty)$, or, with the equivalent $1$-point functional code $C_L(D,(q^3+q^2-q-2-3m)P_\infty)$. Assume that $m$ satisfies $$q^2-q-2\leq 3m\leq 2q^2-q-2$$ and define the integers $a,b$ by $0\leq a,b\leq q-1$ by $3m=2q^2-(a+1)q-b-2$. Then the designed minimum distance is $\delta=3m-q^2+q+2$ and the true minimum distance of $C_\Omega(D,3mP_\infty)$ is
\[d_\mathrm{true} =
\left \{ \begin{array}{ll}\delta & \mbox{if $a<b$,} \\ \delta+b & \mbox{if $a\geq b$.} \end{array}\right . \]
The following table contains some values $q$ and $m$ for which our bound is better that the true minimum distance of the compared $1$-point code.
\vspace{0.15cm}
\[\begin{array}{|c|c|l|}
\hline q&\mbox{cond. on $m$} & \mbox{values of $m$ improving the $1$-point min. distances}\\
\hline\hline
5 & 6 \leq m \leq 14 &  7, 8 \\
7 & 14 \leq m \leq 29 &  18 \\
8 & 18 \leq m \leq 39 &  20, 21, 22, 23, 24, 28, 29, 30 \\
9 & 24 \leq m \leq 50 &  24, 25, 26, 32, 33, 41 \\
11 & 36 \leq m \leq 76 &  38, 39, 40, 41, 42, 43, 44, 50, 51, 52, 61, 62, 63 \\
13 & 52 \leq m \leq 107 &  59, 60, 61, 62, 63, 64, 65, 72, 73, 74, 86, 87, 88 \\
16 & 80 \leq m \leq 164 &  88, 89, 90, 91, 92, 93, 94, 95, 96, 105, 106, 107, 108, \\
&& 109, 110, 111, 112, 121, 122, 123, \
124, 138, 139, 140 \\
\hline
\end{array}\]\vspace{0.2cm}

Finally, we compare our result with the Xing-Chen bound \cite[Corollary 2.6]{XC}. Xing and Chen \cite{XC} used probabilistic method to show the existence of certain divisors $G$ for which the differential code $C_\Omega(D,G)$ with $D$ being as in (\ref{defD}) has good parameters. We confront their  results with Theorem \ref{th:main} for small values of $q$. Notice that the results by Xing and Chen are not constructive; they show the existence of an $\mathbb{F}_{q^2}$-rational divisor $G$ such that $\supp D\cap \supp G=\emptyset$, $t=\deg G$, and the code $C_\Omega(D,G)$ has parameters
\[\left[q^3,t+1-\frac{q^2-q}{2},\geq\frac{2q^3+q^2-q-1-2t}{4+\log_q e}\right].\]
\begin{enumerate}[a)]
\item If $(q,m)=(5,7), (5,8)$ or $(7,19)$ then Xing and Chen improve the designed minimum distance $\delta$ by $2,2$, or $1$, respectively. In these cases, Theorem \ref{th:main} improves $\delta$ by $3,3$, and $4$, respectively.
\item If $q=7$ and $m=18$ then the improvement by Xing and Chen is $4$, while Theorem \ref{th:main} gives the true value $d=\delta+6$.
\item If $q=8$ and $m=21$ then the improvement of Theorem \ref{th:main} equals to the improvement by Xing and Chen. However, our method is constructive, givingherm the divisor $G$ explicitly.
\end{enumerate}

\appendix

\section{Program code} \label{app:source}

\begin{flushleft}
\begin{small}
\begin{verbatim}
q:=7;
BaseRing:=PolynomialRing(GF(q^2),["x","y"]);
x:=BaseRing.1; y:=BaseRing.2;

LoadPackage("singular");
SetInfoLevel( InfoSingular, 2 );
GBASIS:= SINGULARGBASIS;
SingularSetBaseRing( BaseRing );
SetTermOrdering( BaseRing, "dp" );
#################################
H:=x^(q+1)-y-y^q;
R:=x*Product(Filtered(GF(q^2),c->not IsZero(c^q+c)),c->y-c);
a:=Z(q^2);; b:=Z(q^6);;
P:=[b^11896,b^108645];
# Check: P is on the Hermitian curve
IsZero(Value(H,[x,y],P));
#################################

T:=a^26*x^3+a^39*x^2*y+a^32*x*y^2+a^45*x^2+a^40*x*y+
    a^18*y^2+a^41*x+a^45*y-a^0;
A:=a^25*x^4+a^7*x^3*y+x^2*y^2+a^10*x*y^3+a^44*y^4+
    a^4*x^3+a^19*x^2*y+a^4*x*y^2+a^9*y^3+a^37*x^2+
    a^2*x*y+a^3*y^2+a^37*x+a^41*y+a^10;

I:=Ideal(BaseRing,[A,H]);;
liftcoeffs:=SingularInterface("lift", [I,R*T], "matrix");;
C:=liftcoeffs[1][1];;

# Check: I(P,C \cap H)=2
# The tangent of H(X,Y) at P is Y=P[1]^q*X-P[2]^q.
# Substitue this in C(X,Y) and show that X=P[1] is
# a double root.
IsPolynomial(Value(C,[y],[P[1]^q*x-P[2]^q])/(x-P[1])^2);
# Check: C vanishes at te infinite point (0,1,0).
# Show that deg(C)=42 and Y^42 is not a monomial of C.
LeadingMonomialOfPolynomial(C,MonomialLexOrdering());
DegreeIndeterminate(C,y);
# Check: The Hermitian curve has 20 affine rational
# points not lying on C(X,Y)=0.
Hermite:=Filtered(Cartesian(GF(q^2),GF(q^2)),
    p->IsZero(Value(H,[x,y],p)));;
Size(Hermite);
Number(Hermite,p->not IsZero(Value(C,[x,y],p)));
\end{verbatim}
\end{small}
\end{flushleft}

\end{document}